\documentclass{article}
\usepackage{amsfonts,latexsym}
\usepackage{amsmath,amssymb}
\usepackage{amsthm}

\DeclareMathOperator{\SLtwo}{SL_2}
\newcommand{\iso}{\simeq}    
\newcommand{\divides}{\mid}
\newcommand{\C}{\mathbb{C}}

\newcommand{\oo}{\mathcal{O}}
\newcommand{\nn}{\mathcal{N}}
\newcommand{\vv}{\mathcal{V}}
\newcommand{\dimc}[1]{\dim_{\C}(#1)} 
\newcommand{\supr}[1]{^{\,#1}}

\newtheorem{theorem}{Theorem}[section]
\newtheorem{proposition}[theorem]{Proposition}
\newtheorem{lemma}[theorem]{Lemma}
\newtheorem{corollary}[theorem]{Corollary}

\theoremstyle{definition}

\title{The invariants of the binary nonic}
\author{Andries E. Brouwer \& Mihaela Popoviciu\thanks{The second author
is partially supported by the Swiss National Science Foundation.}}
\date{August 1, 2009 / revised February 1, 2010}

\begin{document}
\maketitle

\begin{abstract}
We consider the algebra of invariants of binary forms of degree $9$
with complex coefficients, find the 92 basic invariants, give an
explicit system of parameters and show the existence of four more
systems of parameters with different sets of degrees.
\end{abstract}

\section{Introduction}
{\bf Invariants}\\[1pt]
Let $\oo (V_n)^{\SLtwo}$ denote the algebra of invariants of binary
forms (forms in two variables) of degree $n$ with complex coefficients.
This algebra was extensively studied in the nineteenth century, 
and for $n \leq 6$ the structure was clear and a finite basis
(minimal set of generators) was known.
While Cayley (1856)\footnote{See references at the end of this note.}
states that for $n = 7$ there is no such finite basis,
Gordan (1868) proved that $\oo (V_n)^{\SLtwo}$ has a finite basis for all $n$.
After initial work by von Gall (1880, 1888),
the degrees of the basic invariants in the cases $n=7$ and $n=8$ were
found by Dixmier \& Lazard (1986) and Shioda (1967), respectively.
Bedratyuk (2007) gave an explicit basis in the case $n=7$.
Here we consider the case $n=9$, and settle a 130-year-old question
by showing that $\oo (V_9)^{\SLtwo}$ is generated by 92 basic invariants.
The degrees are given in Proposition \ref{i92}.
The rather large computation needed is discussed in
Section \ref{computation} below.
Earlier work on the case $n=9$ was done by Sylvester \& Franklin (1879)
and by Cr\"{o}ni (2002).

\medskip\noindent
{\bf Systems of parameters}\\
A (homogeneous) {\it system of parameters} for a graded algebra $A$
is an algebraically independent set $S$ of homogeneous elements of $A$
such that $A$ is module-finite over the subalgebra generated by the set $S$.
Hilbert (1893) 
showed the existence of a system of parameters
for algebras of invariants, cf.~Proposition \ref{hilbert} below.

In the case $\oo (V_9)^{\SLtwo}$ considered here, Dixmier (1985) 
proved the following.

\begin{proposition}\label{dixmierparams} 
$\oo(V_9)^{\SLtwo}$ has a homogeneous system of parameters
of degrees $4$, $8$, $10$, $12$, $12$, $14$, $16$.
\end{proposition}

\noindent
Dixmier was unable to give an explicit such system.
Here we find an explicit system of parameters for $\oo(V_9)^{\SLtwo}$
with these degrees (Theorem \ref{explhsop}),
and show the existence of systems of parameters for certain
further sequences of degrees (Proposition \ref{fivehsops}).

\medskip\noindent
{\bf Contents}\\[1pt]
Section 2 gives the Poincar\'e series of the invariant ring.
Its coefficients are the dimensions of the graded parts,
and tell us how many independent invariants we need in each degree.
Section 3 gives the (degrees of) the basic invariants, the main result
of this paper. This result follows by a large computation based on
the knowledge of (the degrees of) a system of parameters.
An explicit such system is given in Section 4, and the proof that
this indeed is a system of parameters follows in Section 5.
Other possible sets of degrees for a system of parameters
are discussed in Section 6, and all such sets for the nonic
are determined in Section 7.

\medskip
\noindent
{\bf Acknowledgements}\\[1pt]
The second author thanks Hanspeter Kraft for the many inspiring
and supporting discussions on the topic of this article. 

\section{Invariants and Poincar\'e series}
Let $V_n=\C [x,y]_n$ be the $\SLtwo$-module of binary forms
(homogeneous polynomials in $x$ and $y$) of degree $n$,
on which $\SLtwo$ acts via
$$
g\cdot f(v)=f(g^{-1}v),
$$
for $g\in \SLtwo$, $f\in\C [x,y]$ and $v\in\C^2$. The coordinate ring of
$V_n$, denoted by $\oo(V_n)$, is isomorphic to the polynomial ring
$\C[a_0,\ldots,a_n]$.
The group $\SLtwo$ acts on the coordinate ring $\oo(V_n)$ via the action 
$$
g\cdot j(f)=j(g^{-1}\cdot f),
$$
for $g\in \SLtwo $, $j\in \oo(V_n)$ and $f\in V_n$. An {\it invariant\/} of
$V_n$ is an element $j\in \oo(V_n)$ such that $g\cdot j=j$ for all
$g \in\SLtwo$. The set of elements of $\oo(V_n)$ invariant under the action of
$\SLtwo$ forms the {\it ring of invariants\/} $I := \oo (V_n)^{\SLtwo}$.

This ring of invariants $I$ is graded by degree, so that
$I = \oplus_m I_m$, where $I_m$ is the subspace of $I$
consisting of the invariants that are homogeneous of degree $m$.
The Poincar\'e series (or Hilbert series) of $I$ is the series
$P(t) = \sum_m \dimc{I_m} t^m$.
Already Cayley and Sylvester (\cite{Cay,Sy}) knew how to compute
this Poincar\'e series.
For a modern account, see, e.g., Springer \cite{Spr}.
In our case $(n=9)$ the series is given by
$$P(t) = \frac{a(t)}{(1 - t^4)(1 - t^8)(1 - t^{10})(1 - t^{12})^2
(1 - t^{14})(1 - t^{16})}$$
with
\begin{align*}
a(t) = \phantom{} & 1+t^4+5t^8+4t^{10}+17t^{12}+20t^{14}+47t^{16}+61t^{18}+97t^{20}+\\
& 120t^{22}+165t^{24}+189t^{26}+223t^{28}+241t^{30}+254t^{32}+254t^{34}+\\
& 241t^{36}+223t^{38}+189t^{40}+165t^{42}+120t^{44}+97t^{46}+61t^{48}+\\
& 47t^{50}+20t^{52}+17t^{54}+4t^{56}+5t^{58}+t^{62}+t^{66},
\end{align*}
so that
\begin{align*}
P(t) = \phantom{} & 1+2t^4+8t^8+5t^{10}+28t^{12}+27t^{14}+84t^{16}+99t^{18}+217t^{20}+\\
& 273t^{22}+506t^{24}+647t^{26}+1066t^{28}+1367t^{30}+2082t^{32}+2649t^{34}+\\
& 3811t^{36}+4796t^{38}+6612t^{40}+8228t^{42}+10960t^{44}+13483t^{46}+\\
& 17487t^{48}+21274t^{50}+26979t^{52}+32490t^{54}+40443t^{56}+48242t^{58}+\\
& 59107t^{60}+69885t^{62}+84470t^{64}+99074t^{66}+...
\end{align*}

\section{The basic invariants}
A minimal set of homogeneous generators for the algebra $I$ is called
a set of `basic invariants' or basis. Such a set is not unique,
but whenever there is a reference to a basic invariant we mean a member
of such a set, fixed in that context.
Let $J_m$ be the subspace of $I_m$ generated by products of invariants
of smaller degree, that is, in $\bigcup_{j<m} I_j$.
The number of basic invariants of degree $m$ is
$d_m := \dimc{I_m/J_m}$.

\begin{proposition}\label{i92}
The algebra $I$ of invariants for the binary nonic (form of degree 9)
is generated by $92$ invariants. The nonzero numbers $d_m$ of
basic invariants of degree $m$ are

\medskip
\begin{center}\begin{tabular}{c|ccccccccc}
$m$ & $4$ & $8$ & $10$ & $12$ & $14$ & $16$ & $18$ & $20$ & $22$ \\
\hline
$d_m$ & $2$ & $5$ & $5$ & $14$ & $17$ & $21$ & $25$ & $2$ & $1$ \\
\end{tabular}\end{center}
\end{proposition}

Finding a basis for the invariants is a simple but boring procedure:
For each degree $m$, multiply invariants of lower degrees
to see what part of $I_m$ is known already. The Poincar\'e series
tells us how large $I_m$ is, and if the known invariants do not yet span it,
one finds in some way some more invariants, until they do span.

This procedure terminates. Gordan \cite{Go} shows that the algebra $I$
is generated by finitely many of its elements.
Better, we know when to stop.
By Proposition \ref{dixmierparams},
$I$ has a system of parameters of degrees 4, 8, 10, 12, 12, 14, 16.
Let $H$ be the ideal in $I$ generated by such a system of parameters.
Now the Poincar\'e series tells us that if
$a(t) = \sum a_i t^i$ then $\dimc{I_i / (I_i \cap H)} = a_i$,
and, in particular, that $I_i \subseteq H$ for $i > 66$.
This means that $d_m = 0$ for $m > 66$.
We followed this procedure, and found the stated values for $d_m$.
These values agree with those given in \cite{Cr} for $m \le 20$.
The existence of a basic invariant of degree 22 was new.

This `finding more invariants in some way' was done by generating
random bracket monomials\footnote{For the classical concept of
bracket monomial, cf.~\cite{Olver}.}. Explicit bracket monomials
for a set of basic invariants are listed in \cite{aeb}.
Checking whether the invariants known span $I_m$
required computing a basis for vector spaces of dimension at most
$\dimc{I_{66}} = 99074$. That is large but doable.
The entire computation can be done in less than a month.

\subsection{Remarks on the computation}\label{computation}
People usually describe invariants in terms of repeated transvectants.
An advantage of working with bracket monomials is that one can simplify
the computations by substituting small constants for a few variables.
This does not work in the approach using transvectants since there one
needs derivatives with respect to the variables.

Given a candidate set for the basic invariants one wants to find
$\dimc{I_m}$ monomials in these basic invariants that span $I_m$.
Since $\dimc{I_m}$ is known, this amounts to the computation of a rank.
The elements involved are far too large to write down. Instead, the
computation is done lazily, and enough coefficients are written down
to find the desired lower bound on the rank.

Also the integer coefficients are far too large, but it suffices to
consider the reduction mod $p$ for some smallish prime $p$,
say with $100 < p < 255$. Now the rank computation of matrices
with sizes like $100000 \times 160000$ just fits within 16 GB of memory.
The generators took a few TB of disk space. Since this problem is still
too large for the standard computer algebra systems, we implemented
our own software (in C, on a Linux system).
Advantage was taken of the presence of multiple CPUs.

This was about the nonics, the case $n = 9$. The difficulty of this
problem grows very quickly with $n$ (and moreover, this computation
cannot be done in a realistic time when the matrices involved are
much larger than main memory). However, the case $n = 2$ (mod 4) is easier,
and $n = 0$ (mod 4) is much easier than the cases of nearby odd $n$.
And indeed, we were able to do the case of decimics ($n = 10$) as well.
For the time being, the case $n = 12$ is still far too large.

\section{A system of parameters for $\oo(V_9)^{\SLtwo}$}

Dixmier \cite{Di1} proved that
the invariant ring of $V_9$ has a system of parameters
of degrees 4, 8, 10, 12, 12, 14, and 16. We compute
an explicit system of parameters of $\oo(V_9)^{\SLtwo}$
having these degrees.

A {\it covariant of order m and degree d} of $V_n$ is an
$\SLtwo$-equivariant homogeneous polynomial map
$\phi: V_n \rightarrow V_m$ of degree $d$
such that $\phi (g \cdot f)=g\cdot \phi (f)$ for all
$g\in \SLtwo$ and $f \in V_n$. The invariants of $V_n$ are the covariants
of order 0. The identity map is a covariant of order $n$ and degree 1.
Customarily, one indicates such a covariant $\phi$ by giving its
image of a generic element $f \in V_n$. (In particular, the identity map
is noted $f$.)
Let $V_{m,d}$ be the space of covariants of order $m$ and degree $d$.

The simplest examples of covariants are obtained using
{\em transvectants\/}: given $g \in V_m$ and $h \in V_n$
the expression
$$
(g,h) \mapsto (g,h)_p:=\frac{(m-p)!(n-p)!}{m!n!}
\sum_{i=0}^p (-1)^i \binom{p}{i}
\frac{\partial ^p g}{\partial x^{p-i}\partial y^i}
\frac{\partial ^p h}{\partial x^i \partial y^{p-i}}
$$
defines a linear and $\SLtwo$-equivariant map
$V_m\otimes V_n\rightarrow V_{m+n-2p}$, which is classically called
the {\it p-th transvectant} (\"Uberschiebung) (cf.~\cite{Olver}).
We have $(g,h)_0=gh$ and $(g,g)_{2i+1}=0$ for all integers $i\geq 0$.
These maps are the components of the Clebsch-Gordan isomorphism (for $m\geq n$)
\[
V_m \otimes V_n \iso V_{m+n}\oplus V_{m+n-2}\oplus \ldots \oplus V_{m-n}.
\]
These maps induce maps $V_{m,d} \otimes V_{n,e} \rightarrow V_{m+n-2p,d+e}$.

For $f \in V_9$, consider the following covariants
\begin{alignat*}{2}
l &= (f,f)_8\in V_{2,2},\quad\quad\quad&
r &= (q,f)_6\in V_{3,3},\\
q &= (f,f)_6\in V_{6,2},&
p &= (f,l)_2\in V_{7,3},\\
u &= (f,f)_2\in V_{14,2},&
k_q &= (q,q)_4\in V_{4,4},
\end{alignat*}
and invariants (the suffix indicates the degree)
\begin{alignat*}{4}
j&_4 &\,=\,& (l,l)_2,&
B&_8 &\,=\,& (q,r^2)_6,\\
j&_{12} &\,=\,& ((k_q,k_q)_2,k_q)_4,\quad&
B&_{12} &\,=\,& ((p,p)_4,l^3)_6,\\
j&_{14} &\,=\,& (q,(r^3,r)_3)_6,&
D&_{10} &\,=\,& ((((u,u)_{10},f)_6,(q,f)_2)_5,q)_6,\\
j&_{16} &\,=\,& ((p,p)_2,l^5)_{10}.
\end{alignat*}

\begin{theorem}\label{explhsop}
The seven invariants $j_4$, $B_8$, $D_{10}$, $B_{12}$, $j_{12}$,
$j_{14}$, $j_{16}$ form a homogeneous system of parameters for
the ring $\oo(V_9)^{\SLtwo}$ of invariants of the binary nonic.
\end{theorem}

This is proved below (\S\ref{theproof}) by invoking Hilbert's
characterization of homogeneous systems of parameters as sets
that define the nullcone.

\section{The nullcone}

The {\em nullcone} of $V_n$, denoted $\nn(V_n)$, is the set of binary forms
of degree $n$ on which all invariants of positive degree vanish.
It turns out (\cite{Hi2}) that this is precisely the set of binary forms
of degree $n$ with a root of multiplicity $>\frac{n}{2}$.
The elements of $\nn(V_n)$ are called {\em nullforms}.
The nullcone $\nn(V_n\oplus V_m)$ is
the set of pairs $(g,h)\in V_n\oplus V_m$ such that $g$ and $h$ have a
common root of multiplicity $>\frac{n}{2}$ in $g$ and of multiplicity
$>\frac{m}{2}$ in $h$. (In this note, this result can be taken as the
definition of the symbol $\nn(V_n\oplus V_m)$.)

We have the following result, due to Hilbert \cite{Hi2},
formulated for the particular case of binary forms:
\begin{proposition} \label{hilbert} 
For $n \ge 3$, consider $i_1,\ldots ,i_{n-2}\in \oo(V_n)^{\SLtwo}$
homogeneous non-constant invariants of $V_n$.
The following two conditions are equivalent:
\begin{itemize}
\item[(i)] $\nn(V_n)=\vv(i_1,\ldots ,i_{n-2})$,
\item[(ii)] $\{i_1,\ldots ,i_{n-2}\}$ is a homogeneous system of parameters
of $\oo(V_n)^{\SLtwo}$. 
\end{itemize}
\end{proposition}
(Here $\vv(J)$ stands for the vanishing locus of $J$.)\\
In other words, if $i_1,\ldots ,i_{n-2}$ are homogeneous invariants
such that $\nn(V_n)=\vv(i_1,\ldots ,i_{n-2})$, then the ring
$\oo(V_n)^{\SLtwo}$ is a finitely generated module over
$\C[i_1,\ldots ,i_{n-2}]$. But invariant rings of binary forms are
Cohen-Macaulay (\cite{HoRo}), which implies that $\oo(V_n)^{\SLtwo}$
is a free $\C[i_1,\ldots ,i_{n-2}]$-module. Hence the description of
the algebra of invariants of $V_n$ is partly reduced to finding a
system of parameters of $\oo(V_n)^{\SLtwo}$.

We prove Theorem \ref{explhsop} by first finding a defining set
for the nullcone that is still too large, and then showing
that some elements are superfluous.

\medskip
\medskip
We need information on the invariants of $V_n$ for $n = 2,\,3,\,6,\,7$:
\begin{lemma}\label{hsops}
The following are systems of parameters of
$\oo(V_n)^{\SLtwo}$ for $n = 2,\,3,\,6,\,7$.
\begin{itemize}
\item[(i)] If $n=2$: $(f,f)_2$ of degree $2$.
\item[(ii)] If $n=3$: $((f,f)_2,(f,f)_2)_2$ of degree $4$.
\item[(iii)] If $n=6$: $(f,f)_6$, $(k,k)_4$, $((k,k)_2,k)_4$,
and $(m^2,(k,k)_2)_4$ of degrees $2$, $4$, $6$, and $10$,
where $k = (f,f)_4$ and $m = (f,k)_4$.
\item[(iv)] If $n=7$: $(l,l)_2$, $((p,p)_4,l)_2$,
$((k_q,k_q)_2,k_q)_4$, $((p,p)_2,l^3)_6$,
$(m_q\supr{2},(k_q,k_q)_2)_4$
of degrees $4$, $8$, $12$, $12$, and $20$,
where $l = (f,f)_6$, 
$p = (f,l)_2$, 
$q = (f,f)_4$, 
$k_q = (q,q)_4$, 
$m_q = (q,k_q)_4$. 
\end{itemize}
\end{lemma}
\begin{proof}
This is classical for $n=2$, $3$, $6$, see, e.g., \cite{Clebsch,GrYo,Schur}.
Systems of parameters for $n=7$ were given by Dixmier \cite{Di0}
and Bedratyuk \cite{Be}.
The above system was constructed by the second author (unpublished).
That it is a system of parameters
can  be easily verified using the methods of this section.
\end{proof}

\begin{lemma}\label{jerzy} {\rm (Weyman \cite{We1})}
Let $f\in V_d$. If $d > 4k-4$ and all $(f,f)_{2k}$, $(f,f)_{2k+2}$, ...
vanish, then $f$ has a root of multiplicity $d-k+1$. If $d = 4k-4$ and
$((f,f)_{2k-2},f)_d$, $(f,f)_{2k}$, $(f,f)_{2k+2}$, ... vanish,
then $f$ has a root of multiplicity $d-k+1$. \qed
\end{lemma}

\begin{lemma} \label{nullform}
Let $f \in V_9$ and consider its covariants $l=(f,f)_8$, $q=(f,f)_6$,
$p=(f,l)_2$, and $r=(f,q)_6$.  
\begin{itemize}
\item[(i)] If $l \neq 0$ and $(l,p)\in \nn(V_2\oplus V_7)$, then $f$ has
a root of multiplicity $5$.
\item[(ii)] If $l=0$, $q\neq 0$ and $(q,r)\in \nn(V_6\oplus V_3)$ then $f$ has
a root of multiplicity $6$.
\item[(iii)] If $l=q=0$, then $f$ has a root of multiplicity $7$.
\end{itemize} 
\end{lemma}
\begin{proof}
Let $f = \sum_{i=0}^9 \binom{9}{i} a_i x^{9-i}y^i$.

\medskip\noindent
(i). From $(l,p)\in \nn(V_2\oplus V_7)$ it follows that both $l$ and $p$
are nullforms and have a common root of multiplicity 2 in $l$ and
4 in $p$. Without loss of generality we suppose $l=x^2$. Then:
$$
p=(f,x^2)_2=\frac{1}{72}\sum_{i=2}^{9}\binom{9}{i}i(i-1) a_ix^{9-i}y^{i-2},
$$
and $x^4$ must divide $p$, which implies $a_6=a_7=a_8=a_9=0$.
Now
$$
l=(f,f)_8=70 a_5\supr{2} y^2+ 28 a_4 a_5 xy + (70 a_4\supr{2} - 112 a_3 a_5)x^2,
$$
and as we suppose $l=x^2$ we also obtain $a_5=0$ and then it follows that
$x^5 \divides f$, so $f$ will have a root of multiplicity 5. 

\medskip\noindent
(ii). From $(q,r)\in \nn(V_6\oplus V_3)$ it follows that both $q$ and $r$
are nullforms and have a common root of multiplicity 4 in $q$ and 2 in $r$.
Without loss of generality we consider the
following 3 cases: $q=x^6$, $q=x^5y$, and $q=x^4y(x+y)$.

\noindent
Case 1: $q=x^6$. Then
$$
r=(f,x^6)_6=a_9y^3+ 3 a_8xy^2+3 a_7x^2y+ a_6x^3,
$$
and $x^2$ must divide $r$.
We obtain $a_9=a_8=0$ and substitute that in $q$ and $l$:
\begin{align*}
q=(f,f)_6=\, &(-20 a_6\supr{2} + 30 a_5 a_7)y^6 + (-30 a_5 a_6 + 54 a_4 a_7)xy^5+\\
& (-90 a_5\supr{2} + 114 a_4 a_6 - 12 a_3 a_7)x^2y^4+\\
& (-72 a_4 a_5 + 124 a_3 a_6 - 60 a_2 a_7)x^3y^3+\\
& (-90 a_4\supr{2} + 114 a_3a_5 - 12 a_2 a_6 - 18 a_1 a_7)x^4y^2+\\
& (-30 a_3 a_4 + 54 a_2 a_5 - 30 a_1 a_6 + 6 a_0 a_7)x^5y+\\
&(-20 a_3\supr{2} + 30 a_2 a_4 - 12 a_1 a_5 + 2 a_0 a_6)x^6,\\
l=(f,f)_8=\,&(70 a_5\supr{2} - 112 a_4 a_6 + 56 a_3 a_7)y^2+\\
& (28 a_4 a_5 - 56 a_3 a_6 + 40 a_2 a_7)xy+\\
& (70 a_4\supr{2} - 112 a_3 a_5 + 56 a_2 a_6 - 16 a_1 a_7)x^2.
\end{align*}
Since we suppose $q=x^6$ and $l=0$, the coefficients of $x^i y^{6-i}$ in $q$
and of $x^j y^{2-j}$ in $l$ are 0 for $0 \le i \le 5$ and $0 \le j \le 2$.

If $a_7=0$ then it follows that $a_6=a_5=a_4=0$ and then $x^6 \divides f$,
so $f$ will have a root of multiplicity 6. 
If $a_7 \neq 0$ then
\begin{align*}
& a_5=\frac{2 a_6\supr{2}}{3 a_7}, ~~
a_4=\frac{10 a_6\supr{3}}{27 a_7\supr{2}}, ~~
a_3=\frac{5 a_6\supr{4}}{27 a_7\supr{3}},\\
& a_2=\frac{7 a_6\supr{5}}{81 a_7\supr{4}}, ~~
a_1=\frac{28 a_6\supr{6}}{729 a_7\supr{5}}, ~~
a_0=\frac{4 a_6\supr{7}}{243 a_7\supr{6}},
\end{align*}
but then we have $q=0$, contrary to the assumption.

\medskip\noindent
Case 2: $q=x^5y$. Then 
\[r=(f,x^5y)_6=-a_8y^3-3 a_7 xy^2-3 a_6 x^2y -a_5 x^3\]
and $x^2$ must divide $r$. We obtain $a_8=a_7=0$ and substitute
this in $q$ and $l$:
\begin{align*}
q=(f,f)_6=\,&(-20 a_6\supr{2} + 2 a_3 a_9)y^6 + (-30 a_5 a_6 + 6 a_2 a_9)xy^5+\\
& (-90 a_5\supr{2} + 114 a_4 a_6 + 6 a_1 a_9)x^2y^4+ \cdots + \\
& (-90 a_4\supr{2} + 114 a_3a_5 - 12 a_2 a_6)x^4y^2+\\
& (-30 a_3 a_4 + 54 a_2 a_5 - 30 a_1 a_6 )x^5y+ \cdots \\
l=(f,f)_8=\,&(70 a_5\supr{2} - 112 a_4 a_6 + 2 a_1 a_9)y^2+ \cdots
\end{align*}
Since we supposed $q=x^5y$ and $l=0$, the coefficient $c$ of $y^2$ in $l$,
and the coefficients $d_i$ of $x^i y^{6-i}$ in $q$ vanish for
$0 \le i \le 4$, while $d_5 \ne 0$. Now
$$5d_5a_9 = -75a_4d_0 + 45a_5d_1 - a_6(9c + 22d_2) = 0$$
so that $a_9 = 0$, and then also $a_6=a_5=a_4=0$, $d_5 = 0$,
contradicting $d_5 \ne 0$.

\medskip\noindent
Case 3: $q=x^4y(x+y)$. Then: 
$$
r=(f,x^4y(x+y))_6=(a_7 - a_8)y^3+3(a_6 -a_7) xy^2+
3(a_5 - a_6)x^2y+(a_4 - a_5)x^3
$$
and $x^2$ must divide $r$. We obtain $a_8=a_7=a_6$ which we
replace in $q$ and $l$:
\begin{align*}
q=(f,f)_6=\,&-2 (6 a_4 a_6 - 15 a_5 a_6 + 10 a_6\supr{2} - a_3 a_9)y^6-\\
    & -6 (5 a_3 a_6 - 9 a_4 a_6 + 5 a_5 a_6 - a_2 a_9)xy^5-\\ 
    & -6 (15 a_5\supr{2} + 3 a_2 a_6 + 2 a_3 a_6 - 19 a_4 a_6 - a_1 a_9)x^2y^4-\\
    & -2 (36 a_4 a_5 - 3 a_1 a_6 + 30 a_2 a_6 - 62 a_3 a_6 - a_0 a_9)x^3y^3-\\
    & -6 (15 a_4\supr{2} -19 a_3 a_5 - a_0 a_6 + 3 a_1 a_6 +2 a_2 a_6)x^4y^2-\\
    & -6 (5 a_3 a_4 - 9 a_2 a_5 - a_0 a_6 + 5 a_1 a_6)x^5y-\\
    & -2 (10 a_3\supr{2} - 15 a_2 a_4 + 6 a_1 a_5 -a_0 a_6)x^6,\\
l=(f,f)_8=\,& 2 (35 a_5\supr{2} - 8 a_2 a_6 + 28 a_3 a_6 - 56 a_4 a_6 + a_1 a_9)y^2+\\
& 2 (14 a_4 a_5 - 7 a_1 a_6 + 20 a_2 a_6 - 28 a_3 a_6 + a_0 a_9)xy+\\
& 2 (35 a_4\supr{2} - 56 a_3 a_5 +a_0 a_6 - 8 a_1 a_6 + 28 a_2 a_6)x^2.   
\end{align*}
As we supposed $q=x^4y(x+y)$ and $l=0$, the coefficients of $y^6$, $xy^5$,
$x^2y^4$, $x^3y^3$, $x^6$ in $q$ and all coefficients of $l$ must vanish.
We denote by $I$ the ideal generated by these coefficients.
Also, we denote by $p_1$, $p_2$ the coefficients of $x^4y^2$ and $x^5y$ in $q$:
\begin{align*}
& p_1=15 a_4\supr{2} -19 a_3 a_5 - a_0 a_6 + 3 a_1 a_6 +2 a_2 a_6,\\
& p_2=5 a_3 a_4 - 9 a_2 a_5 - a_0 a_6 + 5 a_1 a_6.
\end{align*}
A Gr\"obner basis computation shows that $p_1^4$, $p_2\supr{2} \in I$
so that $p_1$ and $p_2$ vanish, contradicting the assumption $q=x^4y(x+y)$. 

\medskip\noindent
(iii). This is a consequence of Lemma \ref{jerzy}.
\end{proof}

\begin{lemma} \label{lpv9}
Let $g \in V_2$ and $h \in V_7$ be two non-zero binary forms.
If both $g$ and $h$ are nullforms and if 
\[((h,h)_6,g)_2=((h,h)_4,g^3)_6=((h,h)_2,g^5)_{10}=(h^2,g^7)_{14}=0,\]
then $(g,h)\in \nn(V_2\oplus V_7)$. 
\end{lemma}
\begin{proof}
Suppose that $(g,h) \notin \nn(V_2\oplus V_7)$. This means that
$g$ and $h$ have no common root which has multiplicity 2 in $g$
and multiplicity 4 in $h$. Without loss of generality we suppose
\begin{align*}
& g=x^2,\\
& h=y^4(b_1 x^3+b_2x^2 y+b_3x y^2+b_4y^3).
\end{align*}
We have then
\begin{align*}
0&=((h,h)_6,g)_2=-\frac{4}{245} b_1\supr{2} ,\\
0&=((h,h)_4,g^3)_6=\frac{2}{735}(5 b_2\supr{2} - 12 b_1 b_3),\\
0&=((h,h)_2,g^5)_{10}=-\frac{2}{147} (3 b_3\supr{2} - 7 b_2 b_4),\\
0&=(h^2,g^7)_{14}=b_4\supr{2}
\end{align*}
and it follows that $b_1=b_2=b_3=b_4=0$, which implies $h=0$.
This contradicts the assumption that $h \ne 0$.
\end{proof}

\begin{lemma} \label{qrv9} Let $g \in V_6$, $h\in V_3$ be two non-zero
binary forms. If both $g$ and $h$ are nullforms and if 
$$
((g^2,g)_6,h^2)_6=((\!(g,g)_2,g)_1,h^4)_{12}=(g,h^2)_6=
(g,(h,h)_2\supr{3})_6=(g,(h^3,h)_3)_6\!=0
$$
then $(g,h)\in \nn(V_6\oplus V_3)$. 
\end{lemma}
\begin{proof}
Suppose that $(g,h) \notin \nn(V_6\oplus V_3)$. This means that
$g$ and $h$ have no common root which has multiplicity 4 in $g$
and multiplicity 2 in $h$.
Without loss of generality we consider two cases:
\begin{align*}
& g=x^4(b_1 x^2+ b_2xy+b_3 y^2),\\
& h=y^3
\end{align*}
and  
\begin{align*}
& g=x^4(b_1 x^2+ b_2xy+b_3 y^2),\\
& h=xy^2.
\end{align*}

\noindent 
Case 1: $h=y^3$. Then we have:
\begin{align*}
0&=((g^2,g)_6,h^2)_6=\frac{1}{495}b_3\supr{3},\\
0&=(((g,g)_2,g)_1,h^4)_{12}=-\frac{1}{540}b_2(5 b_2\supr{2} - 18 b_1 b_3),\\
0&=(g,h^2)_6=b_1
\end{align*}
and it follows that $b_1=b_2=b_3=0$, which implies $g=0$,
contradicting the assumption $g \ne 0$.

\noindent
Case 2: $h=x y^2$. Then we have:
\begin{align*}
0&=(g,h^2)_6=\frac{1}{15}b_3,\\
0&=(g,(h,h)_2\supr{3})_6=-\frac{8}{729}b_1,\\ 
0&=(g,(h^3,h)_3)_6=\frac{1}{84} b_2
\end{align*}
and it follows that $b_1=b_2=b_3=0$, which implies $g=0$,
contradicting the assumption $g \ne 0$.
\end{proof}

\subsection{Proof of Theorem \ref{explhsop}}\label{theproof}
We consider the following covariants of $V_9$:
{\begin{alignat*}{6}
l&_p&\,=\,&(p,p)_6&\,&\in V_{2,6},\quad&
q&_p&\,=\,&(p,p)_4&&\in V_{6,6},\\
p&_p&\,=\,&(p,l_p)_2&&\in V_{5,9},&
k&_{qp}&\,=\,&(q_p,q_p)_4&&\in V_{4,12},\\
k&_q&\,=\,&(q,q)_4&&\in V_{4,4},&
m&_{qp}&\,=\,&(q_p,k_{qp})_4&&\in V_{2,18},\\
m&_q&\,=\,&(q,k_q)_4&&\in V_{2,6},
\end{alignat*}}
and the following invariants of $V_9$:
\begin{alignat*}{4}
j&_4&\,=\,&(l,l)_2,&
A&_4&\,=\,&(q,q)_6,\\
j&_8&\,=\,&(k_q,k_q)_4,&
A&_8&\,=\,&((p,p)_6,l)_2,\\
j&_{12}&\,=\,&((k_q,k_q)_2,k_q)_4,&
A&_{12}&\,=\,&(l_p,l_p)_2,\\
j&_{14}&\,=\,&(q,(r^3,r)_3)_6,&
A&_{20}&\,=\,&(p^2,l^7)_{14},\\
j&_{16}&\,=\,&((p,p)_2,l^5)_{10},&
A&_{36}&\,=\,&((p_p,p_p)_2,l_p\supr{3})_6,\\
j&_{18}&\,=\,&(((q,q)_2,q)_1,r^4)_{12},\quad\quad&
B&_8&\,=\,&(q,r^2)_6,\\
j&_{20}&\,=\,&(m_q\supr{2},(k_q,k_q)_2)_4,&
B&_{12}&\,=\,&((p,p)_4,l^3)_6,\\
j&_{24}&\,=\,&((p_p,p_p)_4,l_p)_2,&
B&_{20}&\,=\,&(q,(r,r)_2\supr{3})_6,\\
j&_{36}&\,=\,&((k_{qp},k_{qp})_2,k_{qp})_4,&
C&_{12}&\,=\,&((r,r)_2,(r,r)_2)_2,\\
j&_{60}&\,=\,&(m_{qp}\supr{2},(k_{qp},k_{qp})_2)_4,&
D&_{12}&\,=\,&((q^2,q)_6,r^2)_6.
\end{alignat*}
Apply Lemma \ref{hsops} to $l \in V_2$, $r \in V_3$, $q \in V_6$
and $p \in V_7$.
It follows that if $j_4=0$ then $l$ is a nullform,
if $C_{12}=0$ then $r$ is a nullform,
if $A_4=j_8=j_{12}=j_{20}=0$ then $q$ is a nullform, and
if $A_{12}=j_{24}=j_{36}=A_{36}=j_{60}=0$, then $p$ is a nullform.
If we combine this information with Lemma \ref{nullform}, Lemma \ref{lpv9}
and Lemma \ref{qrv9} we obtain that
\begin{align*}
\nn(V_9)&=\vv(j_4,A_4,j_8,A_8,B_8,j_{12},A_{12},B_{12},C_{12},D_{12},j_{14},j_{16},j_{18},j_{20},A_{20},B_{20},\\ &j_{24},j_{36},A_{36},j_{60}).
\end{align*}
This can be improved to the following result: 
 
\begin{proposition} \label{nullsmall}
The nullcone $\nn(V_9)$ is the zero set of the following invariants:
\[\nn(V_9)=\vv (j_4,A_4,j_8,A_8,j_{12},B_{12},j_{14},j_{16},j_{20},A_{20}). \]
\end{proposition}
\begin{proof}

If $j_4=0$ then $l$ is a nullform. 

\medskip\noindent
Case 1: $l=0$.

\noindent
If $A_4=j_8=j_{12}=j_{20}=0$ then $q$ is a nullform. Without loss of
generality we suppose $x^4 \divides q$. Modulo the ideal generated by the
coefficients of $l$ and the coefficients of $x^3y^3,x^2y^4,xy^5,y^6$
in $q$ we have
\[B_8=C_{12}=D_{12}=j_{18}=B_{20}=0.\]
(This was an easy computation in Mathematica.)
From Lemma \ref{nullform} it follows then that if $l=0$ and 
\[A_4=j_8=j_{12}=j_{14}=j_{20}=0,\]
then $f$ is a nullform. 

\medskip\noindent
Case 2: $l=x^2$ (without loss of generality). 

\noindent
Here we have:
\begin{align*}
A_{20}&=a_9\supr{2},\\
j_{16}&=-2 (a_8\supr{2} - a_7 a_9),\\
B_{12}&=2 (3 a_7\supr{2} - 4 a_6 a_8 + a_5 a_9),\\
A_8&=-2 (10 a_6\supr{2} - 15 a_5 a_7 + 6 a_4 a_8 - a_3 a_9).
\end{align*} 
Hence if $A_{20}=j_{16}=B_{12}=A_8=0$,
then $a_9=a_8=a_7=a_6=0$,
and if we combine this with $l=x^2$ we get $a_5=0$ too,
hence $f$ is a nullform. 
\end{proof}

But we are still not in the position to apply Proposition \ref{hilbert}.
For that we have to refine our result even more.

We introduce the covariant $s = (f,f)_4 \in V_{10,2}$ and
the following invariants:
\begin{alignat*}{2}
C&_8&\,=\,\,&((q,q)_4,l^2)_4,\\
D&_8&\,=\,\,&((q,q)_4,(q,s)_6)_4,\\
j&_{10}&\,=\,\,&((p,(f,q)_6)_3,(q,q)_4)_4,\\
A&_{10}&\,=\,\,&((p,(f,q)_6)_3,l^2)_4,\\
B&_{10}&\,=\,\,&(((f,q)_6,(f,s)_6)_3,(s,s)_8)_4,\\
C&_{10}&\,=\,\,&((((s,s)_6,f)_6,(l,f)_2)_3,q)_6,\\
D&_{10}&\,=\,\,&((((u,u)_{10},f)_6,(q,f)_2)_5,q)_6.
\end{alignat*}
The invariants $j_8$, $A_8$, $B_8$, $C_8$, and $D_8$ are linearly independent
and together with $j_4\supr{2}$, $A_4\supr{2}$, $A_4j_4$ generate the vector
space of invariants of degree 8 which is of dimension 8.
(This can be seen, e.g., by a small computation in Mathematica.)
In a similar way it can be seen that the vector space of invariants
of degree 10 is generated by $j_{10}$, $A_{10}$, $B_{10}$, $C_{10}$,
and $D_{10}$. 

Using invariants of degree $\leq 16$ we built a list of 219 monomials
of degree 20, each of them dividing one of the invariants $j_4$,
$A_4$, $j_8$, $A_8$, $B_8$, $C_8$, $D_8$, $C_{10}$ or $D_{10}$, to
which we added
\begin{align*}
B&_{20}=((r,r)_2\supr{3},q)_6 ,\\
C&_{20}=(((r^3,r)_3,q)_4,((f,u)_8,(f,s)_8)_3)_4 .
\end{align*}
Let $I$ be the ring of invariants, and $I_i$ its $i$-th graded part.
We evaluated the monomials at $\dimc{I_{20}} = 217$ random points in $V_9$,
giving as result a matrix of (full) rank 217.
Adding $j_{20}$, $A_{20}$, $j_{10}^2$, $A_{10}^2$,
and $B_{10}^2$ to the list of monomials and repeating
the evaluation step gave (of course) again matrices of rank 217.
From the nullspaces of these matrices we obtained the relations
$$
j_{20},A_{20},j_{10}^2,A_{10}^2,B_{10}^2 \in
(j_4,A_4,j_8,A_8,B_8,C_8,D_8,C_{10},D_{10})
$$
(that is, $B_{20}$ and $C_{20}$ are not needed to span the elements
mentioned).

Using invariants of degree $\leq 20$ we built a list of 3561 monomials
of degree 32, each of them dividing one of the invariants $j_4$,
$B_8$, $D_8$, $C_{10}$, $D_{10}$, $j_{12}$, $B_{12}$, $j_{14}$, or
$j_{16}$. We evaluated the monomials at $\dimc{I_{32}} = 2082$
random points in $V_9$, and this resulted in a matrix of rank 2082.
The rank computations were made modulo 32003, but as we obtained
the maximal rank, these monomials must generate $I_{32}$. It follows that
$$
j_8,A_8,C_8,A_4 \in
\sqrt {(j_4,B_8,D_8,C_{10},D_{10},j_{12},B_{12},j_{14},j_{16})},
$$
and then, combining it with Proposition \ref{nullsmall}, we get
\[\nn(V_9)=\vv (j_4,B_8,D_8,C_{10},D_{10},j_{12},B_{12},j_{14},j_{16}).\] 
In the same way one can show that
\[\nn(V_9)=\vv (A_4,B_8,D_8,C_{10},D_{10},j_{12},B_{12},j_{14},j_{16}).\] 

It remains to remove two elements from one of these two sets of generators.
Since this did not seem easy to do by hand, we reverted to the boring approach,
as follows.
Let $H = (j_4,B_8,D_{10},j_{12},B_{12},j_{14},j_{16})$.
We computed $\dimc{I_i \cap H}$ for $i \le 60$
and found $\dimc{I_{60} \cap H} = 59107 = \dimc{I_{60}}$,
so that $I_{60} \subseteq H$.
But then $H$ contains powers of all invariants of degrees 4, 10, 20,
so that in particular $A_4, C_{10} \in \sqrt {H}$.
Now let $H' = (j_4,A_4,B_8,D_{10},j_{12},B_{12},j_{14},j_{16})$.
We computed $\dimc{I_i \cap H'}$ for $i \le 40$ and found
$\dimc{I_{40} \cap H'} = 6612 = \dimc{I_{40}}$, so that
$I_{40} \subseteq H'$.
But then $H'$ contains powers of all invariants of degree 8,
so that in particular $D_8 \in \sqrt {H'}$.
But then $\sqrt {H} = \sqrt {H'} = I$.
Thus,
\[\nn(V_9)=\vv (j_4,B_8,D_{10},j_{12},B_{12},j_{14},j_{16}),\]
and from Proposition \ref{hilbert} it follows that
$\{j_4,B_8,D_{10},j_{12},B_{12},j_{14},j_{16}\}$
is a homogeneous system of parameters of $I$. \qed

\medskip\noindent {\bf Remark}
As a consequence of this result, the proof of Proposition \ref{i92}
no longer requires Proposition \ref{dixmierparams}.
On the other hand, since the end of the proof of the theorem needs
computer work anyway, one can avoid all discussion of the nullcone
following Proposition \ref{hilbert} and show directly that
$\sqrt {H} = I$. From Proposition \ref{i92} we learn that $I$ is
generated by invariants of degrees 4, 8, 10, 12, 14, 16, 18, 20, 22.
Now one can verify that $I_m \subseteq H'$ for $36 \le m \le 44$ and $m = 48$,
hence $\sqrt {H} = \sqrt {H'} = I$.
Thus, Theorem \ref{explhsop} also follows from Dixmier \cite{Di1}
and computer work.

\section{The degrees in a system of parameters}

We give some restrictions on the set of degrees
for the forms in a homogeneous system of parameters (hsop).
Assume $n \ge 3$.

\begin{lemma} Fix integers $j$, $t$ with $t > 0$.
If an invariant of degree $d$ is nonzero on a form $\sum a_i x^{n-i} y^i$
with the property that all nonzero $a_i$ have $i \equiv j$ (mod $t$),
then $d(n-2j)/2 \equiv 0$ (mod $t$).
\end{lemma}
\begin{proof}
For an invariant of degree $d$ with nonzero term $\prod a_i^{m_i}$ we have
$\sum m_i = d$ and $\sum i m_i = nd/2$.
If $i \equiv j$ (mod $t$) when $a_i \ne 0$, then
$nd/2 = \sum i m_i \equiv j \sum m_i = jd$ (mod $t$).
\end{proof}

\begin{lemma} Fix integers $j$, $t$ with $t > 1$ and $0 \le j \le n$.
Among the degrees $d$ of a hsop, at least
$\lfloor (n-j)/t \rfloor$ satisfy $d(n-2j)/2 \equiv 0$ (mod $t$).
\end{lemma}
\begin{proof}
We may suppose $0 \le j < t$.
There are $1 + \lfloor (n-j)/t \rfloor$ coefficients $a_i$
with $i \equiv j$ (mod $t$), so that the subvariety of $V_n$
defined by $a_i = 0$ for $i \not\equiv j$ (mod $t$)
has dimension at least $\lfloor (n-j)/t \rfloor$.
If this is zero, there is nothing to prove. Otherwise,
adding the conditions that the elements of a hsop vanish
reduces this subvariety to a subset of the nullcone.
But the part of this subvariety defined by
$a_i \ne 0$ for $i \equiv j$ (mod $t$) is disjoint from the nullcone.
Indeed, consider the form $a_j x^{n-j}y^j + \cdots + a_{n-k} x^ky^{n-k}$,
where $0 \le j < t$ and $0 \le k < t$ and $j+k \le n-t$ and $a_j$,
$a_{n-k}$ are nonzero but $a_i = 0$ when $i \not\equiv j$ (mod $t$).
The nullcone consists of the forms with a zero of multiplicity
more than $n/2$, but $x=0$ and $y=0$ are zeros of multiplicity $j$ and $k$,
respectively, and if e.g. $j > n/2$, then $k \le n-t-j < n-2j < 0$,
impossible. This means that a zero of multiplicity more than $n/2$
also is a zero of $a_jx^{n-j-k} + \cdots + a_{n-k}$, but this is a
polynomial in $x^t$ and has no roots of multiplicity more than $n/t$.
\end{proof}

\begin{proposition}
Let $t$ be an integer with $t > 1$.

(i) If $n$ is odd, and $j$ is minimal such that $0 \le j \le n$ and
$(n-2j,t) = 1$, then among the degrees of any hsop at least
$\lfloor (n-j)/t \rfloor$ are divisible by $2t$.

(ii) If $n$ is even, and $j$ is minimal with $0 \le j \le \frac{1}{2}n$ and
$(\frac{1}{2}n-j,t) = 1$, then among the degrees of any hsop at least
$\lfloor (n-j)/t \rfloor$ are divisible by $t$. \qed
\end{proposition}
%

\begin{corollary}
Let $t = p^e$ be a power of a prime $p$, where $e > 0$.

(i) Suppose $p=2$. If $n$ is odd, then among the degrees of any hsop
at least $\lfloor n/t \rfloor$ are divisible by $2t$.
If $n/2$ is odd, then at least $\lfloor n/t \rfloor$ degrees
are divisible by $t$. If $4|n$, then at least $\lfloor (n-2)/t \rfloor$
degrees are divisible by~$t$.

(ii) Suppose $p>2$. Among the degrees of any hsop at least
$\lfloor (n-1)/t \rfloor$ are divisible by $t$. \qed
\end{corollary}
%

For example, there exist homogeneous systems of parameters
with degree sequences 4 $(n=3)$; 2, 3 $(n=4)$; 4, 8, 12 $(n=5)$;
2, 4, 6, 10 $(n=6)$; 4, 8, 12, 12, 20 and 4, 8, 8, 12, 30 $(n=7)$;
2, 3, 4, 5, 6, 7 $(n=8)$.

\section{\'Ecritures minimales}
Dixmier \cite{Di0} defines an {\em \'ecriture minimale}
of the Poincar\'e series as an expression $P(t) = a(t) / \prod (t^{d_i}-1)$
with minimal $a(1)$ (or, equivalently, with minimal $\prod d_i$; indeed,
$\underset{t\rightarrow 1}{\lim} (t-1)^{n-2} P(t)= a(1) / \prod d_i$).
He gives the example of $V_7$ where
$P(t) = a(t)/\prod (t^{d_i}-1) = b(t)/\prod (t^{e_i}-1)$
with $d_i = 4,8,12,12,20$ and $e_i = 4,8,8,12,30$,
and there exist systems of parameters of degrees
4, 8, 12, 12, 20 and of degrees 4, 8, 8, 12, 30.

In our case $n = 9$, in view of the restrictions given in the
previous section, the Poincar\'e series can be written
in precisely five minimal ways:

\medskip
\begin{tabular}{cl}
degree $a(t)$ & degrees of factors in denominator \\
\hline
66 & 4, 8, 10, 12, 12, 14, 16 \\
74 & 4, 4, 10, 12, 14, 16, 24 \\
78 & 4, 4, 8, 12, 14, 16, 30 \\
86 & 4, 4, 8, 10, 12, 16, 42 \\
90 & 4, 4, 8, 10, 12, 14, 48 \\
\end{tabular}

\medskip\noindent
and we saw that the first corresponds to a system of parameters.
In fact all five do, as one can show by following the approach
of Dixmier \cite{Di1}.

\begin{proposition} {\rm (Dixmier \cite{Di1})}
Let $G$ be a reductive group over $\C$, with a rational representation
in a vector space $R$ of finite dimension over $\C$. Let $\C[R]$ be the
algebra of complex polynomials on $R$, $\C[R]^G$ the subalgebra of
$G$-invariants, and $\C[R]^G_d$ the subset of homogeneous polynomials
of degree $d$ in $\C[R]^G$. Let $V$ be the affine variety such that
$\C[V] = \C[R]^G$. Let $\delta = \dim V$. Let $(q_1,\ldots,q_\delta)$
be a sequence of positive integers. Assume that for each subsequence
$(j_1,\ldots,j_p)$ of $(q_1,\ldots,q_\delta)$ the subset
of points of $V$ where all elements of all $\C[R]^G_j$ with
$j \in \{j_1,\ldots,j_p\}$ vanish has codimension not less than $p$
in $V$. Then $\C[R]^G$ has a system of parameters of degrees
$q_1,\ldots,q_\delta$.
\end{proposition}

Dixmier gives the covariant $l := (f,f)_8$ and invariants $q_j$ of degree $j$
($j=4,8,10,12,14,16$) such that if $l=0$ and all $q_j$ vanish then
$f$ belongs to the nullcone. It follows that the set of elements in $V$
where $l=0$ and $p$ of the invariants $q_j$ vanish has codimension
not less than $p+1$.

Note that when all invariants of degree $3j$ vanish then also all
invariants of degree $j$ vanish. Therefore, each of the above
five sequences has the property that a subsequence $\sigma$ of length $p+1$
contains at least $p$ distinct elements, and the set of elements in $V$
where $l=0$ and all invariants of the degrees in $\sigma$ vanish has
codimension not less than $p+1$.

Let $[j_1,\ldots,j_p]'$ be the codimension in $V$ of the set of elements
where $l \ne 0$ and all invariants of degrees in $\{j_1,\ldots,j_p\}$
vanish. In order to show that each of the five sequences above is
the sequence of degrees of a system of parameters it suffices to show
that $[4,14]' \ge 3$, $[4,10,14]' \ge 4$, $[4,8,10,14]' \ge 5$,
$[4,8,14,16,30]' \ge 6$, $[4,8,10,16,42]' \ge 6$, given that Dixmier
already proved the requirements of the proposition for the first sequence.

We did this, using instead of `all invariants of degree $j$'
the invariants $p_4,q_4,p_8,p_{10},p_{12},p_{14},p_{16}$ defined
by Dixmier, and moreover $p_{30}$ and $p_{42}$ found by putting
$\tau_1 := (\psi_8,\psi_{10})_0 \in V_{6,10}$,
$\tau_2 := (\psi_8,\psi_{10})_1 \in V_{4,10}$,
$\tau_3 := (\psi_9,\psi_{10})_0 \in V_{6,14}$,
$\tau_4 := (\psi_9,\psi_{10})_1 \in V_{4,14}$,
$p_{30} := ((\tau_1,\tau_1)_4,\tau_2)_4$,
$p_{42} := ((\tau_3,\tau_3)_4,\tau_4)_4$.
The details are very similar to the computation made by Dixmier.
The only less trivial part was to show that $[4,10,14]' \ge 4$,
which was done using the computer algebra system Singular. Thus:

\begin{proposition}\label{fivehsops}
The ring of invariants of $V_9$ has systems of parameters with
each of the five sequences of degrees
$4$, $8$, $10$, $12$, $12$, $14$, $16$ and
$4$, $4$, $10$, $12$, $14$, $16$, $24$ and
$4$, $4$, $8$, $12$, $14$, $16$, $30$ and
$4$, $4$, $8$, $10$, $12$, $16$, $42$ and
$4$, $4$, $8$, $10$, $12$, $14$, $48$. \qed
\end{proposition}

\medskip\noindent
Addresses of authors:

\smallskip
\begin{minipage}{2in}
Andries E. Brouwer \\
Dept. of Math. \\
Techn. Univ. Eindhoven \\
P. O. Box 513 \\
5600MB Eindhoven \\
Netherlands \\
{\tt aeb@cwi.nl}
\end{minipage}
\begin{minipage}{2in}
Mihaela Popoviciu \\
Mathematisches Institut \\
Universit\"at Basel \\
Rheinsprung 21 \\
CH-4051 Basel \\
Switzerland \\
{\tt mihaela.popoviciu@unibas.ch}
\end{minipage}


\begin{thebibliography}{1}

\bibitem{aeb} A. E. Brouwer,
{\tt\verb@http://www.win.tue.nl/~aeb/math/invar.html#nonics@}

\bibitem{Be} L. Bedratyuk,
{\em On complete system of invariants for the binary form of degree 7},
J. Symb. Comput. {\bf 42} (2007) 935--947.

\bibitem{Cay} A. Cayley,
{\em A Second Memoir upon Quantics},
Phil. Trans. Royal Soc. London {\bf 146} (1856) 101--126.

\bibitem{Clebsch} A. Clebsch,
{\em Theorie der bin\"aren algebraischen Formen},
Teubner, Leipzig, 1872.

\bibitem{Cr} H. Cr\"oni,
{\em Zur Berechnung von Kovarianten von Quantiken},
Dissertation, Univ. des Saarlandes, Saarbr\"ucken, 2002.

\bibitem{Di0} J. Dixmier,
{\em S\'erie de Poincar\'e et syst\`emes de param\`etres pour
les invariants des formes binaires de degr\'e 7},
Bull. SMF {\bf 110} (1982) 303--318.

\bibitem{Di1} J. Dixmier,
{\em Quelques r\'esultats et conjectures concernant les s\'eries
de Poincar\'e des invariants de formes binaires.}
pp. 127--160 in:
S\'eminaire d'alg\`ebre Paul Dubreil et Marie-Paule Malliavin 1983--1984,
Springer LNM 1146, Berlin 1985.

\bibitem{DiLa} J. Dixmier \& D. Lazard,
{\em Le nombre minimum d'invariants fondamentaux
pour les formes binaires de degr\'e 7},
Portug. Math. {\bf 43} (1986) 377--392.
Also J. Symb. Comput. {\bf 6} (1988) 113--115.

\bibitem{vG7} F. von Gall,
{\em Das vollst\"andige Formensystem der bin\"aren Form 7ter Ordnung},
Math. Ann. {\bf 31} (1888) 318--336.

\bibitem{vG8} F. von Gall,
{\em Das vollst\"andige Formensystem einer bin\"aren Form achter Ordnung},
Math. Ann. {\bf 17} (1880) 31--51, 139--152, 456. 

\bibitem{Go} P. Gordan,
{\em Beweis, dass jede Covariante und Invariante einer bin\"aren Form
eine ganze Funktion mit numerischen Coeffizienten einer endlichen Anzahl
solcher Formen ist},
Journ. f. Math. {\bf 69} (1868) 323--354.

\bibitem{GrYo} J. H. Grace \& A Young,
{\em The Algebra of Invariants},
Cambridge, 1903.


\bibitem{Hi2} D. Hilbert,
{\em \"Uber die vollen Invariantensysteme},
Math. Ann. {\bf 42} (1893) 313--373.

\bibitem{HoRo} M. Hochster \& J. L. Roberts,
{\em Rings of invariants of reductive groups acting on regular rings
are Cohen-Macaulay}, Adv. Math. {\bf 13} (1974) 115--175.

\bibitem{Olver} P. J. Olver,
{\em Classical Invariant Theory},
Cambridge, 1999.

\bibitem{Schur} I. Schur,
{\em Vorlesungen \"uber Invariantentheorie},
Springer, 1968.

\bibitem{Sh} T. Shioda,
{\em On the graded ring of invariants of binary octavics},
Amer. J. Math. {\bf 89} (1967) 1022--1046.

\bibitem{Spr} T. A. Springer,
{\em Invariant theory},
Springer LNM 585, Berlin, 1977.

\bibitem{Sy} J. J. Sylvester,
{\em Proof of the hitherto undemonstrated fundamental theorem of invariants},
Philos. Mag. {\bf 5} (1878) 178--188.

\bibitem{SyFr} J. J. Sylvester \& F. Franklin,
{\em Tables of generating functions and groundforms for the binary quantics
of the first ten orders},
Amer. J. Math. {\bf 2} (1879) 223--251.

\bibitem{We1} J. Weyman,
{\em Gordan ideals in the theory of binary forms},
J. Algebra {\bf 161} (1993) 370--391.

\end{thebibliography}
\end{document}